\documentclass[a4paper,12pt]{article}
\usepackage{authblk}
\usepackage{hyperref}
\usepackage{enumerate}
\usepackage{amsmath,amssymb,amsthm}
\usepackage{color}
\definecolor{violet}{rgb}{1,0.5,0}
\usepackage{algorithm, algorithmicx, algpseudocode}
\usepackage{graphicx}
\graphicspath{{../fig/}}
\usepackage{fancyvrb}
\usepackage{geometry}
\newtheorem{theorem}{Theorem}

\usepackage[font={small,it}]{caption}

\DeclareMathAlphabet{\mathbf}{OT1}{cmr}{bx}{it}

\newcommand{\vb}{\mathbf b}

\newcommand{\vd}{\mathbf d}
\newcommand{\vdtil}{\widetilde{\vd}}
\newcommand{\ve}{\mathbf e}
\newcommand{\vf}{\mathbf f}

\newcommand{\vq}{\mathbf q}
\newcommand{\vr}{\mathbf r}

\newcommand{\vv}{\mathbf v}
\newcommand{\vw}{\mathbf w}
\newcommand{\vx}{\mathbf x}
\newcommand{\vy}{\mathbf y}
\newcommand{\vz}{\mathbf z}

\newcommand{\Hunder}{\underline{H}}

\newcommand{\spK}{\mathcal K} 

\newcommand{\C}{\mathbb C}

\newcommand{\R}{\mathbb R}

\DeclareMathOperator{\diag}{diag} %

\DeclareMathOperator{\Span}{span} %

\DeclareMathOperator{\sinc}{sinc} %
\newcommand{\divdif}{\mathord{\kern.43em{\vrule width.6pt height7pt depth-.28pt} \kern-.41em\Delta}}

\renewcommand{\d}{\,\mathrm{d}}

\newcommand{\inv}{{-1}}

\newcommand{\one}{{(1)}}
\newcommand{\two}{{(2)}}
\renewcommand{\k}{{(k)}}
\newcommand{\kless}{{(k-1)}}
\newcommand{\kmore}{{(k+1)}}

\newcommand{\norm}[1]{\left\lVert#1\right\rVert}

\newcommand{\changed}[1]{#1}

\begin{document}

\title{Limited-memory polynomial methods\\ for~large-scale matrix functions}

\author{Stefan G\"uttel}
\affil{The University of Manchester, Department of Mathematics,
	Alan Turing Building,
	Manchester M13 9PL, United Kingdom}
\author{Daniel Kressner}
\affil{\'Ecole Polytechnique F\'ed\'erale de Lausanne,
	MA B2 514,
	Station 8,
	1015 Lausanne, Switzerland}
\author{Kathryn Lund\thanks{Corresponding author.}}
\affil{Charles University,
	Faculty of Mathematics and Physics,
	Department of Numerical Mathematics,
	Sokolovská 83,
	186 75 Prague 8, Czech Republic}

\maketitle

\abstract{Matrix functions are a central topic of linear algebra, and problems requiring their numerical approximation appear increasingly often in scientific computing. We review various limited-memory methods for the approximation of the action of a large-scale matrix function on a vector. Emphasis is put on polynomial methods, whose memory requirements are known or prescribed a priori. Methods based on explicit polynomial approximation or interpolation, as well as restarted Arnoldi methods, are treated in detail. An overview of existing software is also given, as well as a discussion of challenging open problems.}

\section{Introduction} \label{sec:introduction}

This survey is concerned with approximating the product of a matrix function $f(A)$, for a function $f$ of a large matrix $A\in\C^{N\times N}$, with a nonzero vector $\vb\in \C^N$. If $A$ is diagonalizable---that is, if there is an invertible matrix $X$ such that $A = X \text{diag}(\lambda_1,\ldots,\lambda_N) X^{-1}$, with the $\lambda_i$ denoting the eigenvalues of $A$---then the matrix function $f(A)$ is defined as 
\begin{equation} \label{eq:diagfA}
	f(A) = X \text{diag}(f(\lambda_1),\ldots,f(\lambda_N)) X^{-1}.
\end{equation}
\changed{Let us stress that diagonalizability is not needed to define matrix functions~\cite{Higham2008}.}
Assuming that $f:\Omega \to \C$ is analytic on a domain $\Omega \subset \C$ containing the spectrum $\Lambda(A) = \{\lambda_1,\ldots,\lambda_N\}$, \changed{one can define
$f(A)$ for general matrices $A$ via the Cauchy integral representation}
\begin{equation} \label{eq:contour}
	f(A) = \frac{1}{2\pi \mathrm i} \int_\Gamma f(z) (zI-A)^{-1}\,\mathrm{d}z,
\end{equation}
where $\Gamma \subset \Omega$ is a contour winding around $\Lambda(A)$ once.  Popular examples of matrix functions arise from the scalar functions $f(z) = z^{-1}$ (leading to the matrix inverse $f(A) = A^{-1}$), $\exp(z)$ (the matrix exponential), $z^{1/2}$ (the matrix square root), and $\mathrm{sign}(z)$ (the matrix sign function).

Just as the matrix inverse $A^{-1}$ is usually not needed explicitly in the context of solving linear systems of equations, it is more common to find the need for evaluating $f(A) \vb$, i.e., the action of a matrix function on a vector, rather than the generally dense matrix~$f(A)$. In Section~\ref{sec:applications}, we briefly discuss a number of applications leading to $f(A)\vb$ with possibly very large matrices~$A$, such as time-dependent partial differential equations (PDEs), problems in quantum chromodynamics,  and graph centrality measures. In this survey, we focus on situations where \emph{the only admissible operations with $A$ are matrix-vector products}.  This is the case, for example, when $A$ is sparse, but its size and sparsity pattern make it very expensive or even impossible to use sparse direct solvers \cite{duff2017direct}. Another frequently encountered case is that $A$ is actually a large dense matrix and cannot be stored explicitly, but there is additional structure that allows for fast matrix-vector products. A prominent example is the fast multipole method for discretized integral operators~\cite{Coifman1993}. If, on the other hand, linear system solves with $A$ are feasible, rational Krylov methods might be efficient alternatives to the polynomial methods discussed here. A separate review of rational Krylov methods for approximating $f(A)\vb$ is given in \cite{Guettel2013}.

This survey discusses algorithms for $f(A)\vb$ that are well suited for large dimensions~$N$ by only requiring a very limited amount of memory. We will focus on the following two main classes of such algorithms, treated in Sections~\ref{sec:expansions} and~\ref{sec:krylov_restarts}.

\begin{paragraph}{Class 1: Explicit polynomial approximation discussed in Section~\ref{sec:expansions}.}
Given a polynomial $p_m(z) = \alpha_0 + \alpha_1 z + \alpha_2 z^2 + \cdots + \alpha_m z^m\in\mathcal{P}_{m}$, we have
\begin{equation} \label{eq:explicitpolynomial}
 p_m(A) \vb = \alpha_0 \vb + \alpha_1 A \vb + \alpha_2 A^2 \vb + \cdots + \alpha_m A^m \vb.
\end{equation}
The evaluation of $p_m(A)\vb$ can be performed within $m$~matrix-vector products involving~$A$. This evaluation becomes particularly convenient when $p_m$ can be written as a sum of polynomials admitting a short-term recurrence, such as Chebyshev polynomials.  In order to use this technique, the given function $f$ needs to be approximated by a suitable polynomial $p_m$. It is already clear from~\eqref{eq:diagfA} that the approximation error $f-p_m$ should be small on a region containing the eigenvalues of $A$. Hence, such an approach requires a priori information on the spectrum of $A$ such as, for example, bounds on the smallest and the largest eigenvalues when $A$ is Hermitian. Care must be taken not only to construct a good polynomial approximant $p_m \approx f$ but also to perform the numerical evaluation of $p_m(A)\vb$ in a stable and efficient manner. On the positive side, methods based on explicit polynomial approximation are often simple to implement and well suited for parallel and distributed computing.
\end{paragraph}

\begin{paragraph}{Class 2: Restarted Krylov methods discussed in  Section~\ref{sec:krylov_restarts}.} \label{para:KSM}
The Krylov space of order $m$ associated with $(A,\vb)$ is defined as 
\[
\spK_{m}(A,\vb) := \Span\{\vb, A\vb, \ldots, A^{m-1}\vb \}\subseteq \mathbb{C}^N,
\]
Equivalently, $\spK_m(A,\vb)$ consists of the vectors $p(A) \vb$ for \emph{every} polynomial $p\in \mathcal{P}_{m-1}$. While this shows the close connection to polynomial approximation, there is a significant difference: Krylov methods select their  approximants $\vf_m\approx f(A)\vb$ from $\spK_m(A,\vb)$ by means of projection, \changed{with neither} a priori specification of the polynomial $p$ satisfying $\vf_m = p(A)\vb$ \changed{nor knowledge of the spectrum of $A$}. This process requires the construction of a basis for $\spK_m(A,\vb)$. For linear systems, solvers like CG and BiCGStab one can avoid the full use of such a basis by constructing simultaneously short-term recursions for the approximate solutions $\vf_m \approx A^{-1} b$~\cite{Saad2003}. Such short-term recursions are not available for general matrix functions, and other techniques are needed to limit the growing memory requirements for storing the basis of $\spK_m(A,\vb)$ as $m$ increases.  In Section~\ref{sec:krylov_restarts}, we describe two stable restarted Krylov approaches for matrix functions, which require only a limited amount of memory while (typically) still producing convergent approximations.
\end{paragraph}

\begin{paragraph}{Illustration of the key differences between the two classes.}
In order to motivate the above classifications, we give a brief discussion to differentiate between the methods. The generation of the orthonormal bases required by Krylov methods, in particular computing the required inner products for orthogonalization, may incur a high computational cost compared to explicit polynomial approximation, which does not require inner products. This aspect is particularly pertinent in distributed computing, e.g., on wireless sensor networks~\cite{Shuman2018}. Also, the storage requirements of explicit polynomial methods are generally lower than those of Krylov methods. On the other hand, Krylov methods can exhibit significantly faster convergence; i.e., they can attain the desired accuracy with a lower degree polynomial, even when compared to excellent choices for an explicit polynomial approximant. This effect, sometimes called ``spectral deflation,'' arises particularly with matrices that have well separated clusters of eigenvalues. 

We illustrate the spectral deflation effect in Figure~\ref{fig:approx}, where the left plot shows the two-norm error $\| f(A)\vb - p_m(A)\vb\|_2$ for degree~$m=1,2,\ldots,15$ approximants computed with two explicit  polynomial approximation methods (uniform and Chebyshev) and a Krylov method (full Arnoldi approximation).  The function to be approximated is $f(z) = z^{-1/2}$, and the matrix $A$ is \changed{ diagonal with equispaced eigenvalues in the interval $[1,10]$ and two outlier eigenvalues close to $z=16$ (with the eigenvalue positions indicated by the grey vertical lines). The vector $\mathbf{b}$ is chosen as the vector of all ones. The \emph{uniform} method corresponds to evaluating $p_m(A)\mathbf{b}$ for the best uniform approximant $p_m$ to $f$ on the spectral interval of $A$, i.e., the unique degree-$m$ polynomial that minimizes $\max_{z\in [1,16]} | f(z) - p_m(z)|$. (The polynomial $p_m$ was computed by the Remez algorithm \cite[Chapter~10]{trefethen2013approximation}.) The Chebyshev method corresponds to evaluating an interpolant $p_m$ to $f$ at $m+1$ Chebyshev points on $[1,16]$; this method is discussed in Section~\ref{subsec:approx} and Figure~\ref{fig:chebfun}. The Krylov approximation corresponds to the (unrestarted) Lanczos method, which is discussed in Section~\ref{sec:krylov_restarts}. Looking at Figure~\ref{fig:approx}, we first} note how the convergence of the Krylov method accelerates superlinearly as the degree $m$ exceeds $7$. This behavior is explained with the plot on the right-hand side showing the scalar error function $|f(z) - p_7(z)|$ of the polynomials of degree $7$ underlying each of the three approximants. While both the ``uniform'' and ``Chebyshev'' polynomials are (almost) uniformly small on $A$'s spectral interval $[1,16]$, the Krylov polynomial attains a particularly small error in regions where $\Lambda(A)$ is located, which is a smaller set than $[1,16]$. In particular the two outliers close to $z=16$ appear to be interpolated almost exactly (the error curve has values close to zero there). Spectral deflation effects are well understood for the classical CG method (a polynomial Krylov method with $f(z)=z^{-1}$), and they also arise and have been analyzed in the Krylov approximation of general matrix functions \cite{beckermann2001superlinear, BeckermannGuettel2012}.

\end{paragraph}

\captionsetup{width=.87\textwidth}
\begin{figure}
\hspace*{0mm}
\includegraphics[scale=.53]{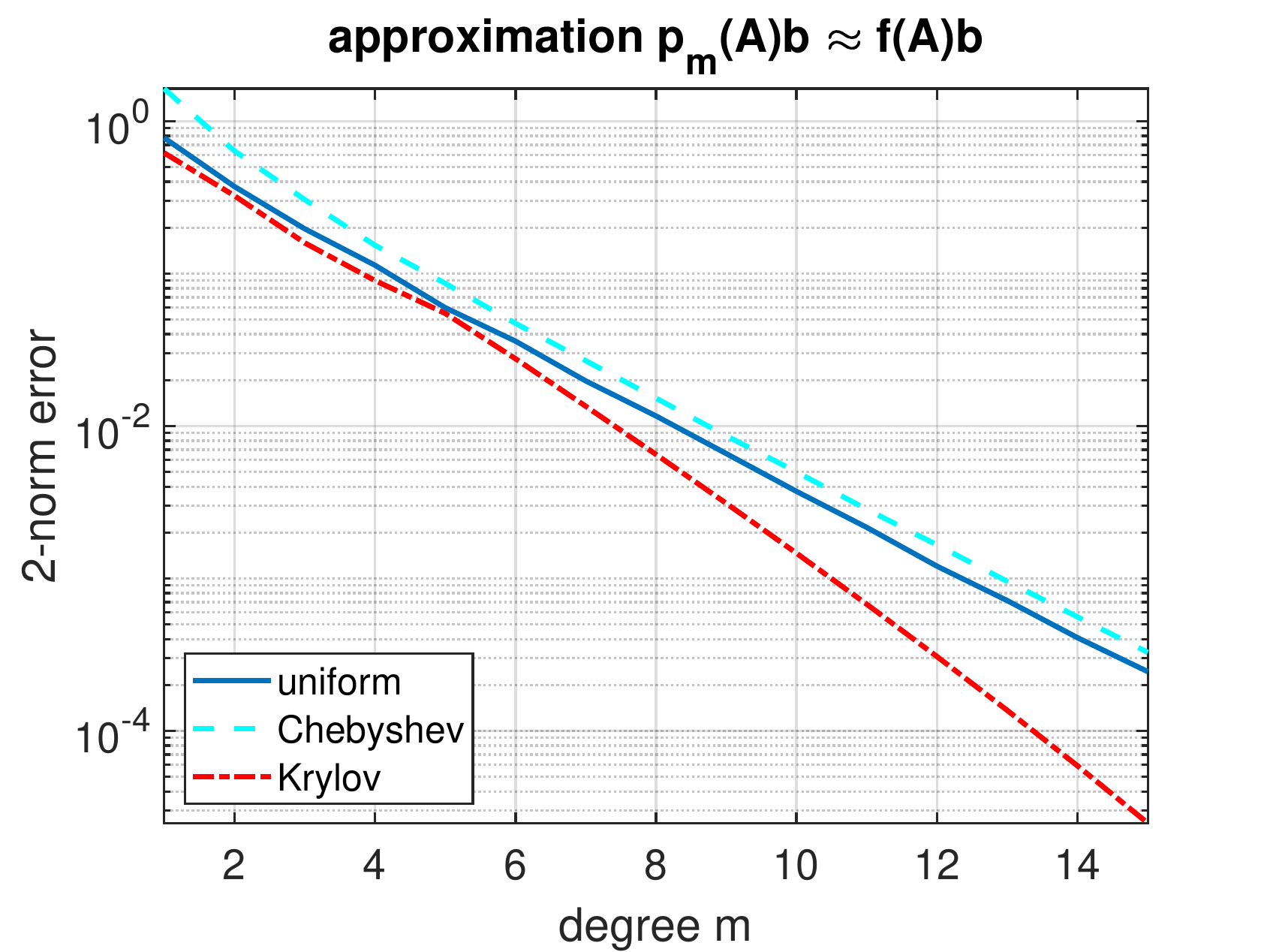}
\includegraphics[scale=.53]{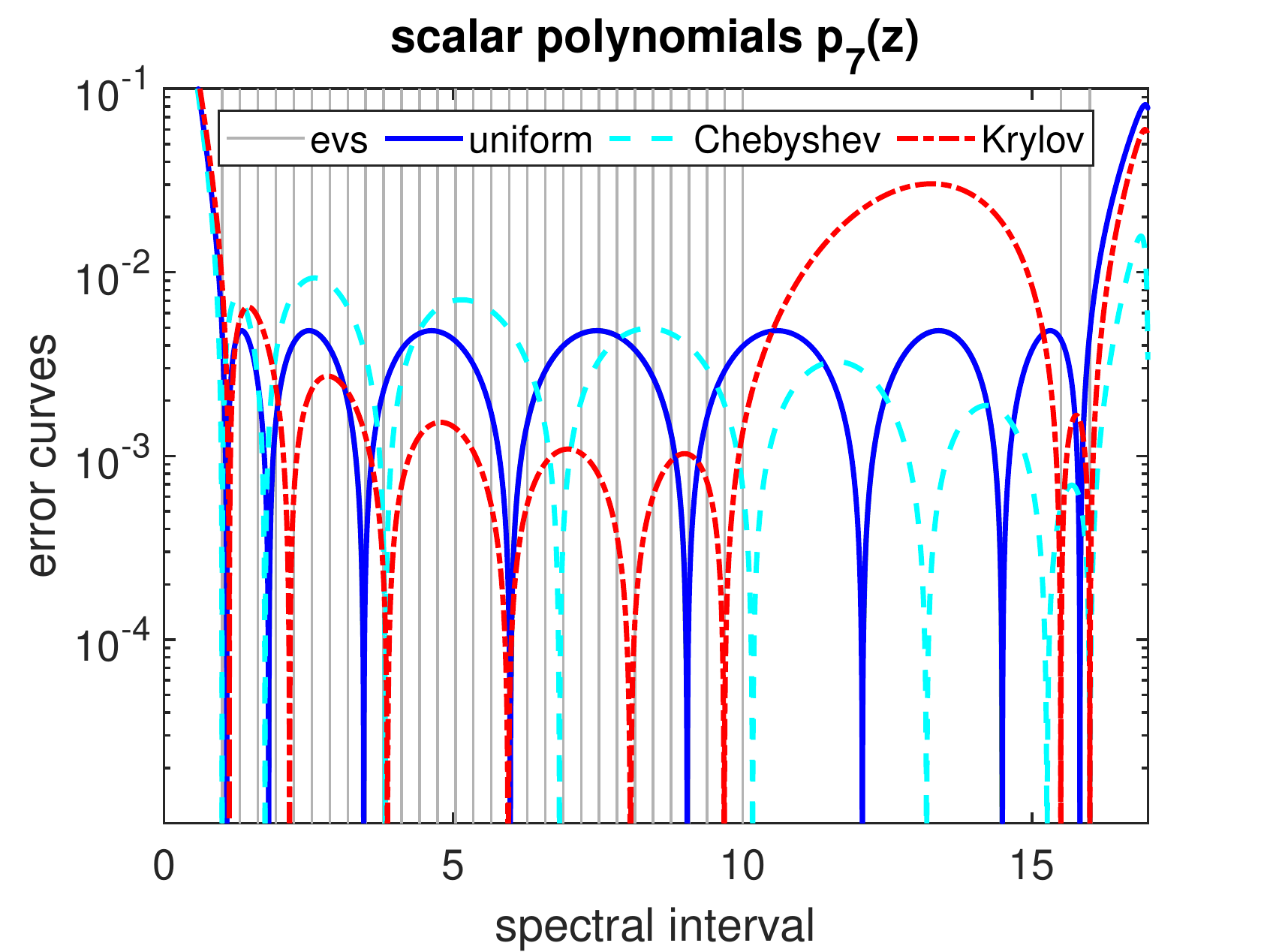}
\caption{Left: Convergence of three polynomial methods for approximating $A^{-1/2}\vb$ for a matrix $A$ with spectral interval $[1,16]$. Right: Error curves $|z^{-1/2}-p_m(z)|$ of the polynomials $p_m$ associated with the approximants of degree $m=7$. The grey vertical lines show the positions of the eigenvalues of $A$. \label{fig:approx}}
\end{figure}

\begin{paragraph}{Further reading.} Necessarily, this survey only covers a limited scope of the existing literature on matrix functions. The monograph by Higham~\cite{Higham2008} gives an authoritative and comprehensive overview of methods for computing $f(A)$ for small- to medium-sized matrices $A$.
The survey by Frommer and Simoncini~\cite{FrommerSimoncini2008a} puts an emphasis on methods for the matrix exponential and the matrix sign function.
Hochbruck and Ostermann~\cite{HochbruckOstermann2010} discuss various methods for $f(A)\vb$ in the context of exponential integrators for (discretized) time-dependent PDEs. The survey by G\"uttel \cite{Guettel2013} covers rational Krylov methods. The recent PhD theses by Schweitzer~\cite{Schweitzer2015} and Lund~\cite{Lund2018} develop algorithms and frameworks for quadrature-based restarting methods and block versions thereof, respectively. The survey by Higham and Al-Mohy \cite[Section~5 and~8]{HighamAlMohy2010} contains some material on the $f(A)\vb$ problem and also makes the distinction between ``a priori approximation'' and Krylov methods. An overview of software for computing matrix functions can be found in~\cite{HighamHopkins2020}; see also the ends of Sections~\ref{sec:expansions} and \ref{sec:krylov_restarts} of this paper. In this issue of GAMM Mitteilungen, Benzi and Boito \cite{benziBoito} and Stoll \cite[Section~5]{stoll} discuss (generalized) matrix functions in the context of network centrality measures \changed{and other data science applications}. 
\end{paragraph}

\section{Selected applications} \label{sec:applications}
We focus on applications that give rise to extremely large matrices, particularly matrices $A$ for which the solution of a linear system $A \vx = \vb$ is itself challenging.  Such matrices include those resulting from discretizations of three-dimensional partial differential equations (PDEs), matrices that are dense but highly structured, or matrices that are so large they can only be stored and accessed in a distributed manner. For a more comprehensive list of matrix function applications, see, e.g., Chapter 2 of Higham\cite{Higham2008}. 

\begin{paragraph}{Differential equations.}
Time-dependent PDEs are a classical and frequent source of large-scale matrix functions. For example, the spatial discretization of the instationary heat equation on a bounded domain $\Omega \subset \R^d$ by a finite difference method leads to a system of ordinary differential equations of the form
\begin{equation} \label{eq:laplace}
{\vx'}(t) = -A \vx(t), \quad \vx(t_0) = \vx_0,
\end{equation}
with a symmetric positive definite matrix $A$. The solution of~\eqref{eq:laplace} admits the explicit representation $\vx(t) = \exp(-tA)\vx_0$. The matrix~$A$ represents the discretization of the (unbounded) Laplace operator on $\Omega$, and as such it is symmetric, large, and sparse. It also tends to have a wide spectrum, which hampers the convergence of (polynomial) Krylov methods~\cite{HochbruckLubich1997} for the matrix exponential, while rational Krylov methods are immune to this problem~\cite{Guettel2013}. The solution of (shifted) linear systems with $A$ required for the latter is a routine calculation for one- and two-dimensional PDEs ($d=1,2$), but this becomes significantly more challenging for $d =3$ and larger~\cite{duff2017direct}.

There exist many extensions and variations of~\eqref{eq:laplace} that lead to matrix functions. For example, the discretization and linearization of a nonlinear PDEs leads to a semilinear equation
\[
{\vx'}(t) = -A \vx(t) + g(\vx(t)), \quad \vx(t_0) = \vx_0.
\]
Exponential integrators proceed by integrating this equation over a (small) time interval $h>0$ and using quadrature to approximate the integral involving the nonlinearity~$g$. In turn, these methods invoke the exponential as well as the so-called $\varphi$-functions of $-hA$; see~\cite{BotchevKnizhnerman2020, HochbruckOstermann2010, MichelsSobottkaWeber2014} for recent work in this direction.
Fractional PDEs involve matrix functions of the form $A^{-\alpha}$ and $\exp(-hA^\alpha)$ for some rational $\alpha>0$, such as $\alpha = 1/2$; see, e.g., \cite{Bolin2018, Burrage2012, CiegisStarikoviciusMargenov2017, HarizanovLazarovMargenov2018, Novati2014}. Matrix square roots
arise from the discretization of Dirichlet-to-Neumann maps, which are useful in, e.g., dealing with Helmholtz problems in unbounded domains~\cite{Druskin2016}. Fractional operators have also been used to regularize ill-posed problems~\cite{Hochstenbach2011}. The wave and time-dependent Schr\"odinger equations~\cite{Grimm2008,Lubich2008} lead to matrix functions of the form $\sin(hA^{1/2})$, $\cos(hA^{1/2})$, $\exp(-\mathrm{i}h A)$, which are notoriously difficult to approximate for Krylov methods due to the highly oscillatory nature of these functions on the spectral interval.

Other discretizations of PDEs may come with further challenges. Finite element discretizations additionally involve a mass matrix $M$; that is, instead of~\eqref{eq:laplace} one obtains $M {\vx'}(t) = -A \vx(t)$. Mass lumping is a common technique to avoid inversion with $M$ in Krylov methods~\cite{Gander2013}. Boundary element methods or, more generally, the discretization of surface and volume integral equations lead to dense matrices~\cite{Sauter2011}. While the fast multipole method~\cite{Coifman1993} is an established technique for efficiently multiplying such matrices with vectors, solving linear systems is considerably more intricate and needs to utilize hierarchical matrices \cite{Hackbusch2015} or related techniques.

Differential equations on four-dimensional computational domains may arise, for example, from space-time formulations of time-dependent PDEs. Usually, such a formulation  does not lead to matrix functions, but there is an important exception: in lattice quantum chromodynamics (QCD) the solutions of linear systems involving the matrix square root or the matrix sign function are needed \cite{BlochBreuFrommer2009, FrommerLundSzyld2020, Hoelbling2014, EshofFrommerLippert2002}. Except for toy examples, it is a futile attempt to apply sparse direct solvers to such matrices. PDEs of dimensions higher than four, such as the electronic Schr\"odinger equation~\cite{Lubich2008}, require additional approximation techniques, beyond the scope of Krylov methods.
\end{paragraph}

\begin{paragraph}{Applications in data analysis.}
Graph centrality measures~\cite{benziBoito}, signal processing on graphs~\cite{Shuman2013}, and graph learning~\cite{ThanouDongKressner2017} are applications that require the evaluation of matrix functions for the adjacency matrix or the Laplacian of an undirected graph. While they are discussed in more detail elsewhere in this issue\cite{benziBoito,stoll}, let us stress that these data analysis applications often feature complex graphs associated with matrices that are usually not amenable to sparse direct solvers. Indeed, the Graph Signal Processing Toolbox~\cite{perraudin2014gspbox} always uses polynomial approximations (e.g., via the Lanczos method~\cite{susnjara2015accelerated}), even for low-degree rational matrix functions. \changed{Other applications, such as clustering via geometric means~\cite{NIPS2016_6164}, require the evaluation of expressions of the form $f(A^{-1}Bv)$ and it is not clear how polynomial based methods can be used most effectively  for that purpose.}
\end{paragraph}

\begin{paragraph}{Spectral projectors and density.}
Approximations of matrix functions feature prominently in solvers for large-scale eigenvalue problems. Consider a symmetric matrix $A$, and a function $\chi_{[\alpha,\beta]}$ that is $1$ on some interval $[\alpha,\beta] \subset \mathbb R$ and $0$ elsewhere. Then, by the definition~\eqref{eq:diagfA}, the matrix function $\chi_{[\alpha,\beta]}(A)$ is the spectral projector on the invariant subspace associated with the eigenvalues  in the interval $[\alpha,\beta]$. In turn, applying $\chi_{[\alpha,\beta]}(A)$ to a vector filters components not belonging to the invariant subspace of interest. A polynomial approximation of $\chi_{[\alpha,\beta]}$ is used, e.g., to accelerate eigenvalue solvers~\cite{Bekas2008,Pieper2016}. The trace of $\chi_{[\alpha,\beta]}(A)$ equals the number of eigenvalues in the interval, and polynomial approximations~\cite{Lin2016} of such eigenvalue counts can be used to slice the spectrum in a way favorable for parallel computation~\cite{Li2019}. By contour integrals, the discussion can, in principle, be extended to general nonsymmetric matrices.
\end{paragraph}

\section{Expansion-based methods} \label{sec:expansions}
Polynomial expansion-based methods are characterized by the property that the polynomial $p_m\in \mathcal{P}_{m}$ underlying the approximation  $\mathbf{f}_{m} = p_m(A)\vb\approx f(A)\vb$ is specified a priori based on the function $f$ and the spectral properties of the matrix~$A$. (The vector $\vb$ is usually not used for the specification of $p_m$ in these methods.) The main advantages of these methods are their low memory consumption (through the use of short-term recurrences), easy implementation and parallelization (as no inner products are required), straightforward extension to actions on block vectors, and the availability of explicit a priori error bounds. On the other hand, expansion-based methods require some advance information about the spectral properties of $A$, and they do not exhibit spectral adaptivity or benefit from deflation like the Krylov   methods discussed in Section~\ref{sec:krylov_restarts}.

In this section we review the short-term recurrences underlying the most commonly used expansion-based methods (Section~\ref{subsec:recur}), followed by a discussion of the relevant approximation theory to construct the underlying polynomials  (Section~\ref{subsec:approx}). We then give a brief overview of convergence properties (Section~\ref{subsec:conv}) and conclude with a list of available software (Section~\ref{subsec:soft}).

\subsection{Short-term recurrences}\label{subsec:recur}
Expansion-based methods utilize explicit polynomial approximants of functions. Of particular interest are polynomial approximants that can be written in the form $p_m(A)\vb = \sum_{j=0}^{m} \gamma_j \vq_j$ with the vectors $\vq_j$ generated by a short-term recurrence. For example, if 
\[
	q_{0}(z) :\equiv 1, \quad q_1(z) := \alpha_1(z), \quad \text{and}\quad  q_{j}(z) := \alpha_j(z) q_{j-1}(z) - \beta_j q_{j-2}(z), \quad j=2,\ldots,m
\]
is a short-term recurrence with linear nonconstant polynomials $\alpha_j(z)$ and numbers $\beta_j$, this  translates  into a polynomial short-term recurrence 
\[
	\vq_{0} := \vb, \quad \vq_{1} := \alpha_1(A)\vb, \quad \vq_{j} := \alpha_j(A) \vq_{j-1} - \beta_j \vq_{j-2}.
\]
This forward recursion can be intertwined with the summation to form $p_m(A)\vb$, a procedure that is often referred to as Forsythe's algorithm \cite{forsythe1957generation}. Alternatively, the summation can be performed in reverse direction using Clenshaw's algorithm \cite{clenshaw1955note}:
\[
	\mathbf{c}_{m+1} := \mathbf{0}, \quad \mathbf{c}_{m}:=\gamma_m \mathbf{b}, 
\]
\[
	\mathbf{c}_{j} := \gamma_j\mathbf{b} + \alpha_{j+1}(A)\mathbf{c}_{j+1} - \beta_{j+2} \mathbf{c}_{j+2}, \quad j = m-1,\ldots,0.
\]
Eventually, $\mathbf{c}_0 = p_m(A)\mathbf b$. Both algorithms  require the same number of $m$ matrix-vector products and limited memory as at most three or four vectors need to be stored at any time, respectively. A computational advantage of using  Forsythe's algorithm over Clenshaw's is that one can synchronously sum polynomials $p_m^{(k)}(A)\vb$ with different coefficients~$\gamma_j^{(k)}$, enabling the approximation of multiple matrix functions $f^{(k)}(A)\vb$ without increasing the number of matrix-vector multiplications. Both summation algorithms generalize straightforwardly to the block case, that is, to the approximation of $f(A) B$, where $B \in \mathbb{C}^{N\times s}$ is a ``block'' of $s$ vectors.

Clenshaw's algorithm has been analyzed by Smoktunowicz~\cite{smoktunowicz2002backward} for the evaluation of scalar polynomials and found to be backward stable under natural conditions on the coefficients $\alpha_j,\beta_j$ and the argument~$z$, at which the recursion is evaluated.  A forward stability analysis and comparison of both algorithms for the summation of scalar polynomial series is given by Barrio~\cite{barrio2002rounding}. Among other cases, he considers the Chebyshev polynomials $q_j = T_j$ of the first kind, defined as 
\[
	T_0(z) = 1, \quad T_1(z) = z, \quad T_j(z) = 2z T_{j-1}(z) - T_{j-2}(z), 
\]
and numerically summed for an argument $z\in [-1,1]$ to obtain an approximate evaluation $\widetilde p_m(z)$. It is found that both summation algorithms exhibit a forward error $|p_m(z) - \widetilde p_m(z)|$ of $O(m)$ for arguments near $z=0$, and an error of $O(m^2)$ with Clenshaw and $O(m^3)$ with Forsythe, respectively, near the interval endpoints $|z|=1$. Backward stability has also been verified for the Clenshaw vector recursion by Aurentz et al., \cite[Theorem~4.1]{aurentz2019stable} when the $q_j$ are  Chebyshev polynomials which are scaled and shifted to the spectral interval of a Hermitian matrix. (More precisely, the authors deal with the approximation of generalized matrix functions defined as $f^\diamond(B) :=  U f(\Sigma)V^*$ where $B = U\Sigma V^*$ is a singular value decomposition of $B$. A Chebyshev interpolation method using Clenshaw's algorithm with a matrix argument $z=A=B B^*$ is used.)

\subsection{Approximation and interpolation}\label{subsec:approx}
There are various ways the polynomial $p_m$ can be determined given $f$ and $A$, with the most important approach being scalar \emph{approximation} (e.g., uniform or in least-squares sense) or \emph{interpolation}. If the matrix $A$ is diagonalizable, $A = X \diag(\lambda_1,\ldots,\lambda_N)X^{-1}$, and we assume further that  $\Sigma\subset \mathbb{C}$ is a compact set containing $\Lambda(A)$, then we have a simple bound on the approximation error as follows:
\begin{eqnarray*}
	\| f(A)\vb - p_m(A)\vb\|_2 &\leq& \kappa_2(X)\, \|\vb\|_2 \max_{i=1,\ldots,N} | f(\lambda_i) - p_m(\lambda_i) |\\
	 &\leq& \kappa_2(X)\, \|\vb\|_2 \max_{z\in\Sigma} | f(z) - p_m(z)| \\
	 &=:& \kappa_2(X)\,  \|\vb\|_2 \,\| f-p_m\|_\Sigma,
\end{eqnarray*}
with the two-norm condition number $\kappa_2(X) := \| X \|_2 \| X^{-1} \|_2$. 
More generally, scalar approximation can be related to matrix function approximation if there exists a constant $K>0$ and a compact set $\Sigma\subset\mathbb{C}$ such that $\|p(A)\|_2 \leq K \| p \|_\Sigma$ for all polynomials~$p$. Such a set $\Sigma$ is called a $K$-spectral set \cite{badea14}. It is known~\cite{crouzeix2017numerical} that for an arbitrary square matrix $A$, the numerical range $W(A) := \{ \vv^* A\vv : \| \vv\|_2 = 1\} \subset \mathbb{C}$ is a $K$-spectral set with $K\leq 1+\sqrt{2}$ (conjectured to be $K=2$). Assuming that $f$ is analytic on $\Sigma$, we have
\[
	\|f(A)\vb - p_m(A)\vb \|_2 \leq K\|\vb\|_2  \| f - p_m \|_{W(A)}. 
\]
Hence, in the absence of other explicit knowledge about spectral properties of $A$, it is reasonable to construct $p_m$ as a uniform approximant to $f$ on $\Sigma$. 

Early examples of methods based on expansions of $f$  into orthogonal polynomials on an interval are studied by Tal-Ezer and Kosloff \cite{tal1984accurate}, wherein $f(z) = \exp(z)$  for propagating the time-dependent Schr\"odinger equation ($z\in [ -\mathrm{i}d,\mathrm{i} d]$ with $d>0$), and by Druskin and Knizhnerman \cite{DruskinKnizhnerman1989} for solving constant-coefficient linear homogeneous PDEs  ($z\in [-d,0]$ with $d>0$ for parabolic problems like the heat equation). In both cases, shifted and scaled Chebyshev polynomials $q_j$ are used. More precisely, assuming that $A$ is a normal matrix with eigenvalues in a (possibly complex) interval $[c,d]$, we  consider the linear transform $\ell(z) = (2z-c-d)/(d-c)$ such that $\widehat A := \ell(A)$ has eigenvalues in $[-1,1]$, and then approximate $f(A)\vb = f(\ell^{-1}(\widehat A))\vb$ by a truncated Chebyshev expansion of $\widehat f(z) := f(\ell^{-1}(z))$ on $[-1,1]$:
\[
p_m(A)\vb = \sum_{j=0}^m \gamma_j T_j(\widehat A) \vb \approx f(A)\vb
\]
with Chebyshev coefficients 
\[
	\gamma_j = \frac{2}{\pi} \int_{-1}^1  \frac{\widehat f(z) T_j(z)}{\sqrt{1-z^2}} \d z,
\]
except for $\gamma_0$ where the factor $2/\pi$ is changed to $1/\pi$. For many functions of practical interest, closed formulas for the Chebyshev coefficients $\gamma_j$ are known, including the exponential function, for which they can be given in terms of Bessel functions. The case of Chebyshev expansions of nonanalytic functions has been discussed by Sharon and Shkolnisky~\cite{SharonShkolnisky2018}.

We refer to Trefethen's text\cite{trefethen2013approximation} for an in-depth discussion of the approximation theory of Chebyshev polynomials on intervals. We highlight that, for practical purposes, the use of Chebyshev interpolants instead of expansions might be preferable. The Chebyshev interpolant $\widehat p_m(z) = \sum_{j=0}^m \widehat \gamma_j T_j(z)$ of degree $m$ is defined as the unique interpolating polynomial of $\widehat f$ at the Chebyshev points $x_j = \cos(j\pi/m)$, $j=0,1,\ldots,m$. The coefficients $\widehat \gamma_j$ are readily computed using the fast Fourier transform. We give a basic MATLAB implementation of the resulting $f(A)\vb$ method in Figure~\ref{fig:chebfun}.

\begin{figure}
\hspace*{18mm}
\hspace*{15mm}\begin{verbatim}
      c = 0; d = 40i; m = 35; % spectral interval and degree
      f = @(z) exp(z); A = 10i*gallery('tridiag',100); b = eye(100,1);
      % compute coefficients of Chebyshev interpolant:
      x = .5*(cos((0:m)*pi/m)+1)*(d-c)+c; fx = f(x);
      gam = fft([fx,fx(m:-1:2)])/m; gam(1) = gam(1)/2; gam(m+1) = gam(m+1)/2;
      AA = (2/(d-c))*A - (c+d)/(d-c)*speye(size(A));
      % perform Clenshaw summation to obtain approximant fAb:
      c2 = 0*b; c1 = gam(m+1)*b;
      for j = m-1:-1:0
         if j==0, fAb = gam(j+1)*b + AA*c1 - c2;
         else, c0 = gam(j+1)*b + 2*(AA*c1) - c2; c2 = c1; c1 = c0; end
      end
      norm(fAb - expm(A)*b)/norm(fAb) % 4.2038e-07
\end{verbatim}
\caption{Computation and evaluation of a degree~$m=35$ Chebyshev interpolant $p_m(A)\vb \approx f(A)\vb$ for $f(z)=\exp(z)$ on a spectral interval $[c,d]$.}\label{fig:chebfun}
\end{figure}

Alternatives to the above procedures are discussed, e.g., by Chen et\ al.\cite{ChenAnitescuSaad2011}, who use least-squares approximation to obtain suitable expansion coefficients, and Moret and Novati\cite{MoretNovati2001b}, who use truncated Faber expansions for approximating functions of real nonsymmetric matrices.
\changed{Beckermann and Reichel~\cite{BeckermannReichel2009} analyze such Faber expansions, establishing -- among others -- elegant bounds for the matrix exponential in terms of the angle of the right-most corner of the numerical range. In the absence of such a corner, partial Faber sums may exhibit slow (initial) convergence. When it comes to using these expansion-based procedures in practise, estimates for the numerical range of $A$ are required.}

There are expansion-based algorithms that, instead of approximating $f(A)\vb$ directly, first apply a linear transformation $\widetilde A  = (A - \sigma I)/s$ to $A$,  (repeatedly) compute $\widetilde f(\widetilde A)\widetilde \vb$, and use functional identities to relate the results to $f(A) \vb$. For example, with $s$ being a positive integer, the approximation of $\exp(A)\mathbf b$ can be reduced to $s$~computations of exponentials of $\widetilde A$ using the identity
\[
	\exp(A) \vb = \exp(s\widetilde A + \sigma I) \vb = e^\sigma  [\exp(\widetilde A)]^s  \vb.
\]
Such a splitting is reminiscent of a time-stepping method for the linear ODE $\mathbf{x}' = A\mathbf{x}$, $\mathbf{x}(0)=\mathbf{b}$. It is at the core of various algorithms for the expansion-based approximation of $\exp(A)\vb$, such as that of Al-Mohy and Higham \cite{AlMohyHigham2011}, who use truncated Taylor series, therefore resulting in a method equivalent to an explicit Runge--Kutta method. The selection of the scaling parameter~$s$ and the degree of the Taylor polynomial applied to approximate $\exp(\widetilde A)$ is based on a forward error analysis and a sequence of the form $\|A^k\|^{1/k}$ in such a way that the overall computational cost of the algorithm is minimized. This scaling approach might result in a rather large total number of matrix-vector products roughly proportional to $\|A\|$. Similarly, Al-Mohy\cite{AlMohy2017} derives an algorithm for computing $\cos$, $\sinc$, and other trigonometric matrix functions. Here, the ``unscaling'' is done by using recurrences of the Chebyshev polynomials of the first and second kind. 

Apart from orthogonal polynomials (three-term recurrence) and monomials in Taylor series (one-term recurrence), another class of methods for the approximation of matrix functions is based on Newton polynomial interpolation (two-term recurrence). Early references for this idea include Berman et\ al.~\cite{berman1992solution} and Moret and Novati \cite{MoretNovati2001a} (with interpolation nodes obtained via conformal mapping) and Huisinga et\ al.~\cite{huisinga1999faber} (interpolation at Leja nodes).  The algorithmic aspects of this approach have been investigated and improved in a number of works, in particular for the matrix exponential function \cite{CaliariVianelloBergamaschi2004,CaliariKandolfOstermann2016,
CaliariKandolOstermann2014,KandolfOstermannRainer2014}.

We briefly outline  the basic idea underlying Leja interpolation on some  compact set $\Sigma\subset \mathbb{C}$. Starting with a node $\sigma_0\in\Sigma$ such that $|\sigma_0| = \max_{z\in\Sigma} |z|$, we choose a Leja sequence of nodes such that 
\begin{equation}\label{eq:leja}
\sigma_j \in \mathrm{arg\,max}_{z\in\Sigma} \prod_{k=0}^{j-1} \left| z - \sigma_k \right|, \quad j = 1,2,\ldots
\end{equation}
With these nodes we can associate a sequence of interpolating polynomials $p_0(z) \equiv f(\sigma_0)$,
\[
	p_j(z) = p_{j-1}(z) + f_{j} w_{j-1}(z), \quad j = 1,2,\ldots
\]
with the nodal polynomial $w_j(z) := \prod_{k=0}^{j} ( z - \sigma_k )$ and the divided difference $f_j := (f(\sigma_j) - p_{j-1}(\sigma_j))/w_{j-1}(\sigma_j)$ of the function $f$ at the (distinct) nodes $\sigma_0,\ldots, \sigma_j$. This scheme can be straightforwardly adapted for the interpolatory approximation of $f(A)\vb$ as follows. Define $\mathbf p_0 := f(\sigma_0)\mathbf b$, $\mathbf{w}_0 := \mathbf{b}$, and the recursion
\[
	\mathbf{p}_j :=  \mathbf{p}_{j-1} + f_j \mathbf{w}_{j-1}, \quad \mathbf{w}_j := (A - \sigma_j)\mathbf{w}_{j-1}, \quad j = 1,2,\ldots.
\] 
After $m$ steps, we have $\mathbf{p}_m = p_m(A)\mathbf{b}$. However, the naive implementation of this recursion may result in numerical cancellations and over or underflow. A careful shifting and scaling of the matrix $A$, combined with an appropriate time-stepping procedure (if applicable, as in the case of the matrix exponential), can turn Leja interpolation into an attractive computational method.

\subsection{Convergence analysis}\label{subsec:conv}
Let us consider a polynomial interpolation process with interpolation nodes $\sigma_j$ on a compact set $\Sigma\subset \mathbb{C}$. We also assume that the function $f$ to be interpolated is analytic on an open set $\Omega\supset \Sigma$.
Using again the nodal polynomial $w_m(z) = \prod_{k=0}^{m} ( z - \sigma_k )$, the unique polynomial interpolant $p_m$ of degree $m$ for $f$ at the nodes $\sigma_j$ satisfies
\begin{equation}\label{eq:hermerr}
   f(z) - p_m(z) = \frac{1}{2\pi \mathrm{i}} \int_\Gamma \frac{f(\zeta)}{(\zeta - z)} \frac{w_m(z)}{w_m(\zeta)} \d\zeta,\quad z\in\Sigma,
\end{equation}
where $\Gamma\subset\Omega$ is a contour containing $\Sigma$ in its interior. In order to bound this error uniformly on $\Sigma$, the growth of $w_m$ on $\Sigma$ and $\Gamma$ needs to be studied. If $\mathbb{C}\setminus \Sigma$ is connected, this can be done conveniently using the level lines of the Green's function $G(x,y)$, i.e., the unique harmonic function defined for all $z = x + \mathrm{i} y \in \mathbb{C}\setminus \Sigma$ with a pole at infinity, zero boundary value on $\partial \Sigma$, and normalized such that $\frac{1}{2\pi} \int_{\partial \Sigma} \frac{\partial G}{\partial n} (x,y) \d s = 1$. Reichel \cite[Lemma~2.3]{reichel1990newton} shows that for polynomial interpolants $p_m$ using the Leja nodes defined in \eqref{eq:leja} we have  
\[
	\mathrm{lim\,sup}_{m\to\infty} \| f - p_m \|_\Sigma^{1/m} = e^{-\rho}, 
\]
where $\rho > 0$ is the largest constant such that $f$ is analytic and single-valued  in the interior of the level curve $L_\rho = \{ z = x + iy : G(x,y) = \rho \}$. Moreover, there is no sequence of polynomials $q_m$ such that $\mathrm{lim\,sup}_{m\to\infty} \| f - q_m \|_\Sigma^{1/m} < e^{-\rho}$, hence the linear convergence factor $e^{-\rho}$ is optimal. We say that the Leja interpolants are maximally convergent. The same optimal convergence factor is achieved by best uniform polynomial approximants for $f$ on $\Sigma$, as well as by truncated Faber expansions \cite{MoretNovati2001b}, interpolants in roots of Faber polynomials \cite{MoretNovati2001a} or Fej{\`{e}}r points \cite{Novati2003} (applicable when $\mathbb{C}\setminus\Sigma$ is simply connected). For a comparison of Chebyshev and Leja approximation, see Baglama et\ al.\cite{baglama1998fast}.

If $f$ is an entire function, superlinear convergence will take place. In this case \eqref{eq:hermerr} can also be used to study the uniform  error $\| f - p_m\|_\Sigma$, but now $\Gamma$ needs to be chosen so that the growth of $f$ is balanced with the decay of $1/w_m(\zeta)$ on $\Gamma$. The superlinear convergence bounds of polynomial interpolation for $f(z) = e^z$ in \cite{HochbruckLubich1997} are derived using this technique (see also \cite{CaliariVianelloBergamaschi2004}). 






\subsection{Software}\label{subsec:soft}

Here is a list of available software  for approximating matrix functions $f(A)\vb$ based on expansions. See also the \changed{recent} report by \changed{Higham and Hopkins~\cite{HighamHopkins2020}} for a more comprehensive list.
\begin{itemize}
	\item \texttt{expmv}: \url{https://github.com/higham/expmv}. A \textsc{Matlab} implementation of the Taylor series-based method for $\exp(tA)\vb$ described in \cite{AlMohyHigham2011}.  A Java version can be found at \url{https://github.com/armanbilge/AMH11}.
	\item \texttt{ExpLeja}: \url{https://bitbucket.org/expleja/expleja/src/default/}. A \textsc{Matlab} implementation of a Leja method for $\exp(tA)\vb$, with the details described in \cite{CaliariKandolfOstermann2016}.
	\item \texttt{rcexpmv}: \url{https://www.mathworks.com/matlabcentral/fileexchange/28199-matrix-exponential}. A \textsc{Matlab} implementation of (rational) Chebyshev interpolation of $\exp(A)\vb$, applicable when $A$ is Hermitian negative semidefinite.
\end{itemize}

\section{Restarted Krylov  methods} \label{sec:krylov_restarts}

Krylov methods admit another type of polynomial approximation $p(A)\vb\approx f(A)\vb$, but instead of determining the polynomial~$p$ a priori, they compute the polynomial implicitly based on a basis-building process and the choice of Galerkin, Petrov--Galerkin, or other conditions for extracting an approximation from a Krylov subspace.  \changed{We emphasize again that no spectral knowledge is required a priori.}

Recall the definition for the Krylov space of order~$m$ associated with $(A,\vb)$, 
$\spK_m(A,\vb) := \Span\{\vb, A\vb, \ldots, A^{m-1}\vb \}\subseteq \mathbb{C}^N$. We compute $m$ orthonormal vectors spanning $\spK_{m}(A,\vb)$ and stored as the columns of $V_{m} \in \C^{N \times m}$, ordered such that $\beta V_m\ve_1 = \vb$ where $\beta = \|\vb\|_2$, as well as the projection and restriction of $A$ onto $\spK_m(A,\vb)$, represented by the upper-Hessenberg matrix $H_m = V_m^* A V_m\in \C^{m \times m}$, satisfying
\begin{equation} \label{eq:arnoldi_relation}
A V_m = V_{m+1} \Hunder_m = V_m H_m + \vv_{m+1} h_{m+1,m} \ve_m^T.
\end{equation}
Here, $\ve_1$ and $\ve_m$ refer to the first and $m$-th canonical unit vectors, respectively, and $V_{m+1} = [ V_m , \mathbf{v}_{m+1} ]$ and $\Hunder_m = \begin{bmatrix} H_m \\ h_{m+1,m} \ve_m^T \end{bmatrix} \in \C^{(m+1) \times m}$. 
Equation~\eqref{eq:arnoldi_relation} is referred to as an \emph{Arnoldi decomposition}, after the Arnoldi procedure given as Algorithm~\ref{alg:arnoldi}.

\begin{algorithm}[htbp!]
	\caption{Arnoldi procedure \label{alg:arnoldi}}
	\begin{algorithmic}
		\State Given $A\in\C^{N\times N}$,  $\vb\in\C^N $, integer $m\geq 1$
		\State Compute $\vv_1 \leftarrow \vb/\beta$ with $\beta = \norm{\vb}_2$
		\For{$k = 1, \ldots, m$}
			\State $\vw \leftarrow A \vv_{k}$
			\For{$j = 1, \ldots, k$}
				\State $h_{j,k} \leftarrow \vv_j^* \vw$
				\State $\vw \leftarrow \vw - h_{j,k} \vv_j$
			\EndFor
			\State $h_{k+1,k} \leftarrow \norm{\vw}_2$
			\State $\vv_{k+1} \leftarrow \vw / h_{k+1,k}$
		\EndFor
		\State Return $V_{m+1} = [\vv_1, \vv_2, \ldots , \vv_{m+1}]$ and $\Hunder_m = [ h_{jk} ]$
	\end{algorithmic}
\end{algorithm}

The Arnoldi approximation to $f(A)\vb$ \cite{Saad1992, FrommerGuettelSchweitzer2014a} is defined as
\begin{equation} \label{eq:arnoldi_f(A)b}
\vf_m := V_m f(V_m^* A V_m) V_m^* \vb =  \beta V_m f(H_m) \ve_1 .
\end{equation}
This approximation is also known as the Ritz \cite{HochbruckHochstenbach2005} or Lanczos (if $A$ is Hermitian) \cite{DruskinKnizhnerman1989} approximation, and it can be thought of as a generalization of the full orthogonalization method (FOM) to general functions $f$, since when $f(z) = z^\inv$, the FOM approximation is recovered.  With modifications to $H_m$, a GMRES-like or \emph{harmonic} approximation \cite{HochbruckHochstenbach2005, FrommerGuettelSchweitzer2014b} can be implemented, as well as a \emph{Radau--Lanczos} approximation \cite{FrommerLundSchweitzer2017}, which allows one to prescribe eigenvalues of $H_m$.  \changed{Even more general approaches are possible via Krylov-like and Arnoldi-like decompositions \cite{Stewart2001, EiermannErnst2006, EiermannErnstGuettel2011}, which allow for non-orthogonal bases.  Indeed, this more general framework is useful from a conceptual point of view when developing restarts; cf.\ equations~\eqref{eq:arnoldi-like_relation}-\eqref{eq:f2m_v1}.}

As the basis $V_m$ grows, the Arnoldi approximation $\vf_m \approx f(A)\vb$ typically becomes more accurate.  However, when we are in a limited-memory scenario, it may not be possible to store enough basis vectors to reach the desired accuracy.  Short-term recurrences via the Lanczos procedure are generally not useful, as the entire basis $V_m$ is needed to compute~\eqref{eq:arnoldi_f(A)b}.  A two-pass procedure for Hermitian $A$ is however possible, as long as we have an efficient enough implementation to tolerate the high computational costs of computing the Krylov basis twice.  The procedure is rather simple. On the first pass, only $H_m$ is stored and the coefficient vector $\vy: = \beta f(H_m) \ve_1$ is computed, while on the second pass, the approximation $\vf_m = V_m \vy = \sum_{k=1}^{m} \vv_k y_k$ is computed one summand at a time as the basis is generated anew via the Lanczos process. It is important to note that no (full) reorthogonalization can be performed during this process \changed{(as it often done for linear systems)} because this would require access to the Krylov basis. As discussed in~\cite{DruskinGreenbaumKnizhnerman1998,DruskinKnizhnerman1995}, the lack of reorthogonalization does not jeopardize numerical stability but it may result in delayed convergence.
An alternative solution to memory limitations is to \emph{restart} the Arnoldi process.

Restarted Krylov methods for linear systems are well studied, and some of  their convergence analysis carries over to matrix functions; see Algorithm~\ref{alg:restarts_lin_sys} for an outline of the basic restart procedure for the linear system of equations $A \vx = \vb$. Note that the update $\vz_m^\k$ computed by solving the correction equation $A \vz = \vr_m^\k$ is an approximation to the error $\vx - \vx_m^\k$. The existence of this useful linear correction equation relies on the fact that we are solving a linear system of equations and it will, in general, not be available for general matrix funtions.   It is important to keep in mind that restarts may either speed up or slow down convergence.  For both FOM and GMRES, it is possible to design a pair $(A,\vb)$  for which convergence is as slow as possible or even stagnates \cite{Schweitzer2016a, DuintjerTebbensMeurant2014, VecharynskiLangou2011}.  At the same time, there are scenarios in which a shorter cycle length corresponds counter-intuitively to faster convergence~\cite{Embree2003}.  In light of these issues, it is common to combine restarts with acceleration techniques, such as spectral deflation, ``thick" restarting, and weighted inner products, in order to obtain faster convergence \cite{StathopoulosSaadWu1998, EiermannErnstSchneider2000, WuSimon2000, BakerJessupManteuffel2005, EiermannErnstGuettel2011, YinYin2014, EmbreeMorganNguyen2017}.

\begin{algorithm}[htbp!]
	\caption{Basic restart procedure for  $A\vx=\vb$}\label{alg:restarts_lin_sys}
	\begin{algorithmic}
		\State Given $A\in\C^{N\times N}$, $\vb\in\C^N$, residual tolerance $\varepsilon>0$, and   integers $m,k_{\max}\geq 1$
		\State Compute Krylov approximation $\vx_m^\one$ to $A \vx = \vb$ and residual $\vr_m^\one = \vb - A \vx_m^\one$
		\While {$k = 1, \ldots, k_{\max}$ and  $\norm{\vr_m^\k}_2 > \varepsilon$}
		\State Compute Krylov approximation $\vz_m^\k$ to $A \vz = \vr_m^\k$
		\State Update $\vx_m^\kmore \leftarrow \vx_m^\k + \vz_m^\k$
		\State Compute $\vr_m^{(k+1)} = \vb - A \vx_m^{(k+1)}$
		\EndWhile
	\end{algorithmic}
\end{algorithm}

To formulate a restart procedure for the Arnoldi approximation~\eqref{eq:arnoldi_f(A)b}, we need an  approximation to the error after each cycle, similarly to the role played by $\vz_m^\k$ in Algorithm~\ref{alg:restarts_lin_sys}.  Except for special types of functions $f$, devising such a correction procedure is not straightforward, and a number of techniques have been proposed over the years; see, e.g., Section 2 of Frommer et\ al. \cite{FrommerGuettelSchweitzer2014a}, which compares several approaches \cite{EiermannErnst2006, AfanasjewEiermannErnstGuettel2008b, IlicTurnerSimpson2010} in detail.  In the following, we focus on two approaches that have proven to be numerically stable, with the first one being the most general (Section~\ref{sec:rest1}), but the second one being generally more efficient (Section~\ref{sec:rest2}). Finally, we end with a list of Krylov-based software available for the approximation of $f(A)\vb$ (Section~\ref{sec:rest3}).

\subsection{Restarting for general matrix functions}\label{sec:rest1}
The first approach we consider was developed by Eiermann and Ernst \cite{EiermannErnst2006} and is designed for any function for which $f(H)$ can be directly computed for a small, dense matrix $H \in \C^{M \times M}$, $M \ll N$.  We begin by demonstrating how the procedure works for two Arnoldi cycles.
Assume that after the first cycle we have obtained an Arnoldi decomposition $A V_m^{(1)} = V_m^{(1)} H_m^{(1)} + h_{m+1,m}^{(1)} \vv_{m+1}^{(1)} \ve_m^T$ and the corresponding Arnoldi approximation $\vf_m^{(1)} = \beta V_m^{(1)} f(H_m^{(1)}) \ve_1$. 
We can now compute a second Arnoldi decomposition $A V_m^{(2)} = V_m^{(2)} H_m^{(2)} + h_{m+1,m}^{(2)} \vv_{m+1}^{(2)} \ve_m^T$ for the Krylov space $\spK_m(A,\vv_{m+1}^{(1)})$, i.e., with  $\vv_{m+1}^{(1)} = V_m^{(2)}\ve_1$ being the starting vector. The goal is to approximate the error
\begin{equation} \label{eq:error_EE06}
	\vd_m^\one := f(A) \vb - \vf_m^\one
\end{equation}
with just $V_m^\two$ and the enlarged block-Hessenberg matrix
\[
H_{2m} :=
\begin{bmatrix}
H_m^\one 						& O \\
h_{m+1,m}^\one \ve_1 \ve_m^T 	& H_m^\two
\end{bmatrix}.
\]
To this end we consider a concatenation of the Arnoldi decompositions from $\spK_m(A,\vb)$ and $\spK_m(A,\vv_{m+1}^{(1)})$,
\begin{equation} \label{eq:arnoldi-like_relation}
A W_{2m} = W_{2m} H_{2m} + h_{m+1,m}^\two \vv_{m+1}^\two \ve_{2m}^T,
\end{equation}
where $W_{2m} := [V_m^\one , V_m^\two]$.  This relation~\eqref{eq:arnoldi-like_relation} is often referred to as an Arnoldi-like decomposition. An analysis of these and more general types of Arnoldi decompositions, and discussions of their uses in acceleration techniques for restarts, can be found in Stewart \cite{Stewart2001}, Eiermann and Ernst \cite{EiermannErnst2006}, and Eiermann et\ al.\cite{EiermannErnstGuettel2011}. 

Using the Arnoldi-like decomposition~\eqref{eq:arnoldi-like_relation} we can define an associated Arnoldi-like approximation to $f(A)\vb$ as 
\begin{equation} \label{eq:f2m_v1}
\vf_{2m} :=  \beta W_{2m} f(H_{2m}) \ve_1.
\end{equation}
Since $H_{2m}$ is block triangular, we have (see, e.g., Theorem 1.13 in \cite{Higham2008})
\[
f(H_{2m}) = \begin{bmatrix}
f(H_m^\one) & O \\
X_{2,1} & f(H_m^\two)
\end{bmatrix},
\]
so that \eqref{eq:f2m_v1}  reduces to
\begin{equation} \label{eq:f2m_v2}
\vf_{2m} = \vf_m^\one +  \beta V_m^\two X_{2,1} \ve_1.
\end{equation}
Equation~\eqref{eq:f2m_v2} thus suggests a way to update $\vf_m^\one$, as long as we can determine a computable expression for $X_{2,1}$.  In \cite{EiermannErnst2006}, this is first achieved via divided differences formulas and the characteristic polynomial for $H_m^\two$, and it is shown that $\beta V_m^\two X_{2,1} \ve_1 $ is indeed an approximation to $\vd_m^\one$.  For subsequent restart cycles, however, extending this process based on divided differences leads to numerical instabilities.  It turns out that computing $X_{2,1}$ directly by taking the bottom left block of $f(H_{2m})$ is stable but requires more computational effort, because the entire matrix $f(H_{2m})$ must be computed first. A complete description of this costly, but numerically stable, procedure for an arbitrary number of restart cycles is given by Algorithm~\ref{alg:restarts_mat_func1}, \changed{wherein MATLAB notation is used in the last line to denote the extraction of the last $m$ elements of the computed vector}.  \changed{Indeed, only the first column of $f(H_{km})$ is needed to compute the update; it remains an open problem whether this can be achieved more efficiently than computing all of $f(H_{km})$.}

\changed{We have not specified how to obtain error bounds or estimates for the restarted Arnoldi approximation. A detailed error analysis is outside the scope of this survey but can be found, along with an acceleration procedure based on thick restarting, in \cite{EiermannErnst2006, EiermannErnstGuettel2011}.}

\begin{algorithm}
	\caption{Restarted Arnoldi approximation for general $f(A)\vb$}\label{alg:restarts_mat_func1}
	\begin{algorithmic}
		\State  Given function $f$, $A\in\C^{N\times N}$, $\vb\in\C^N$ of norm $\beta>0$, tolerance $\varepsilon>0$, and   integers $m,k_{\max}\geq 1$ 
		\State Compute an Arnoldi decomposition for $\spK_m(A,\vb)$: $A V_{m}^\one =  V_{m}^\one H_m^\one + h_{m+1,m}^\one\vv_{m+1}^\one \ve_m^T$
		\State Compute first Arnoldi approximation $\vf_m^\one \leftarrow  \beta V_m^\one f(H_m^\one)\ve_1$
		\State Set $H_{1m} \leftarrow H_m^\one$
		\While {$k = 2, \ldots, k_{\max}$ and error measure greater than $\varepsilon$}
		\State Compute an Arnoldi decomposition for $\spK_m \left(A,\vv_{m+1}^\kless \right)$: $A V_{m}^\k = V_{m}^{(k)} H_m^{(k)} + h_{m+1,m}^{(k)}\vv_{m+1}^{(k)} \ve_m^T$
		\State Update $H_{km} \leftarrow \begin{bmatrix} H_{(k-1)m} & O \\ h_{m+1,m}^\kless \ve_1 \ve_{(k-1)m}^T &H_m^\k \end{bmatrix}$
		\vspace*{2mm}
		\State Update $\vf_m^\k \leftarrow \vf_m^\kless + \beta V_m^\k [f(H_{km}) \ve_1]_{(k-1)m+1:km}$
		\EndWhile
	\end{algorithmic}
\end{algorithm}
Note that the matrix $H_{km}$ grows in size with each restart cycle. At some point, its size could become prohibitively large for evaluating $f(H_{km})$, depending on the available computational resources.  The restarting approach reviewed in the next section overcomes this problem, but the scope of functions to which it can be applied  is limited.

\subsection{Quadrature-based restarting for Cauchy--Stieltjes functions}\label{sec:rest2}
Residual relations for families of shifted linear systems can be used to formulate an error update function for matrix functions that have a Cauchy--Stieltjes representation.  Indeed, the following approach can be extended to other functions with some kind of integral representation, but the analysis becomes much more difficult; see, e.g., the work by Frommer et\ al.\cite{FrommerGuettelSchweitzer2014a}.

Let $f: \C \setminus (-\infty,0] \to \C$ be a Cauchy--Stieltjes function of the form
\begin{equation} \label{eq:Cauchy-Stieltjes}
f(z) = \int_0^\infty \frac{1}{z + t} \d \mu(t),
\end{equation}
where $\mu$ is a positive and monotonically increasing, real-valued function on $ [0,+\infty)$ such that $\int_0^\infty \frac{1}{1+t} ~\d\mu(t) < \infty$. Such functions are a subclass of Markov functions \cite{GuettelKnizhnerman2011}.  For more information about properties of these functions, see the text by Henrici \cite{Henrici1977} or the introductions to Ili\'{c} et\ al.\cite{IlicTurnerSimpson2010} and Schweitzer's thesis \cite{Schweitzer2015}.

An important example of a Cauchy--Stieltjes function is the inverse square root.  More generally, for $\alpha \in (0,1)$, we have
\begin{equation} \label{eq:zalpha_int}
z^{-\alpha} = \frac{\sin((1-\alpha)\pi)}{\pi} \int_{0}^{\infty} \frac{1}{z+t} \d \mu(t) \quad \quad \text{with} \quad \d \mu(t) = t^{-\alpha} \d t.
\end{equation}
The logarithm can also be expressed in terms of a Cauchy--Stieltjes integral, after noting that
\begin{equation} \label{eq:logz_int}
\frac{\log(1+z)}{z} = \int_{0}^{\infty} \frac{1}{z + t} \d\mu(t) \quad \text{with} \quad \d \mu(t) =
\begin{cases}
0 \d t, & \text{for\ } 0 \leq t \leq 1, \\
t^{-1} \d t, & \text{for\ } t > 1.
\end{cases}
\end{equation}

With the form \eqref{eq:Cauchy-Stieltjes}, $f(A) \vb$ is defined as $f(A) \vb = \int_0^\infty (A + tI)^\inv \vb \d \mu(t)$, which can alternatively be thought of as an integral over $\vx(t)$, the solution to the family of shifted systems
\begin{equation} \label{eq:family_shifted_systems}
(A + tI) \vx(t) = \vb, \quad t \geq 0.
\end{equation}
Since Arnoldi decompositions \eqref{eq:arnoldi_relation} are \emph{shift-invariant}, i.e., $(A + tI) V_m = V_m(H_m + tI) + h_{m+1,m} \vv_{m+1} \ve_m^T$, it is reasonable to approximate $\vx(t)$ by
\[
\vx_m(t) := \beta V_m(H_m + tI)^\inv \ve_1.
\]
In this way, the same Krylov basis $V_m$ is used for all shifts and needs to be constructed only once.  Consequently, $\vf_m = \int_0^\infty \vx_m(t) \d \mu(t)$.  Letting $\vr_m(t) = \vb - (A + tI) \vx_m(t)$ denote the residual of $\vx_m(t)$, we find that the error for the matrix function approximation can be written as
\begin{equation} \label{eq:error_krylov-like_approx_f(A)b}
\vd_m := f(A) \vb - \vf_m = \int_0^\infty (A+tI)^\inv \vb - \vx_m(t) \d \mu(t) = \int_0^\infty (A + tI)^\inv \vr_m(t) \d \mu(t).
\end{equation}
The idea is to use \eqref{eq:error_krylov-like_approx_f(A)b} and the fact that the residual for \eqref{eq:family_shifted_systems} is \changed{the integrand} to formulate a restart procedure.  We could approximate $\vd_m$ by building the next Krylov space from $\vr_m(t)$ to find an approximation to the shifted error and integrate over $t$.  If we have a quadrature rule set a priori, we would have a finite number of shifts.  \changed{Since Krylov subspaces are shift-invariant, it is possible to use the same space for each shift.}

Frommer and Gl{\"a}ssner \cite{FrommerGlaessner1998}, as well as Simoncini \cite{Simoncini2003} study restarts for GMRES and FOM methods applied to \eqref{eq:family_shifted_systems}, respectively.  In both cases, a \emph{colinear} relationship between $\vr_m(0)$ and $\vr_m(t)$ is obtained,\footnote{For GMRES, the starting approximation must be identical for all shifted systems.} allowing them to perform restarts for shifted systems from a single Krylov space at each new cycle.  Theorem~4.1 from Frommer et\ al. \cite{FrommerLundSzyld2020} comprises both results by formulating a \emph{cospatial} relationship for block Krylov methods. We state here just the result for FOM.
\begin{theorem} \label{thm:residual_colinearity}
	Let $\rho_m(t) := -\beta\ve_m^T(H_m + tI)^\inv \ve_1$. Then
	\[
	\vr_m(t) = \rho_m(t) \vv_{m+1},
	\]
	and in particular, $\vr_m(t)$ is colinear to $\vr_m(0)$ with factor $\eta_m(t):= \rho_m(t)/\rho_m(0)$.
\end{theorem}
\begin{proof}
	The proof follows immediately after substituting \eqref{eq:arnoldi-like_relation} in the definition of the shifted residual.
\end{proof}
With Theorem~\ref{thm:residual_colinearity}, we thus only need to compute the basis for $\spK_m(A,\vv_{m+1})$, which can then be used for all residuals $\vr_m(t)$.  Denoting again quantities in the Arnoldi decompositions associated with $\spK_m(A,\vb)$ and $\spK_m(A,\vv_{m+1})$ with superscripts~$^\one$ and~$^\two$, respectively, we approximate $\vd_m^\one$ by
\[
\vdtil_m^\one := \int_0^\infty \rho_m^\one(t)\, V_m^\two \big( H_m^\two + tI \big)^\inv \ve_1  \d \mu(t).
\]
We can approximate the integral with a quadrature rule, which is now over small systems of size $m$.  The colinear factor $\rho_m^\one(t)$ can also be computed more efficiently if more is known about the structure of $H_m^\one$.  Indeed, it is known that $\rho_m^\one(t)$ is a polynomial whose roots are the Ritz values $\Lambda(H_m^\one)$, and similar results hold for harmonic or Radau--Lanczos approximations \cite{FrommerGuettelSchweitzer2014a, FrommerGuettelSchweitzer2014b, FrommerLundSchweitzer2017}. We then only need to keep the Ritz values from the previous cycle to compute the integral, instead of the whole matrix $H_m^\one$.

It is of course possible to restart repeatedly.  The only caveat is that we must keep track of the colinear factors $\rho_m^\k(t)$, since they accumulate in the integral.  After $k+1$ cycles, one has to compute
\begin{equation} \label{eq:error_integral}
\vdtil_m^\k := V_m^\kmore I_m^\k(H_m^\kmore) \ve_1,
\mbox{\quad where \quad }
I_m^\k(z) := \int_0^\infty \rho_m^\k(t) \cdots \rho_m^\one(t)\, (z + t)^\inv  \d \mu(t),
\end{equation}
To reduce the computational effort of the quadrature approximation for $I_m^\k(z)$, it is recommended to use an adaptive rule that reduces the number of nodes as the integrand becomes smaller (which happens naturally as the error  $\|f(A)\vb - \vf_m^{(k)}\|_2$ decreases). For details on specific choices of quadrature rules for the functions $z^{-\alpha}$, $\log(1 + z)/z$, and $\exp(z)$, see Frommer et\ al.\cite{FrommerGuettelSchweitzer2014a}.

The full procedure is outlined in Algorithm~\ref{alg:restarts_mat_func2}.  Note that it is also possible to make the size of the basis (i.e., $m$) adaptive, as well as the choice of $H_m$ or a harmonic or Radau--Lanczos version of $H_m$, but a detailed analysis of these parameter choices remains open. Error analyses for Algorithm~\ref{alg:restarts_mat_func2} and its variants on Hermitian positive definite or positive real matrices can be found in various works by Frommer et\ al. \cite{FrommerGuettelSchweitzer2014a, FrommerGuettelSchweitzer2014b, FrommerLundSchweitzer2017}.

\begin{algorithm}
	\caption{Quadrature-based restarted Arnoldi approximation for $f(A)\vb$}\label{alg:restarts_mat_func2}
	\begin{algorithmic}
		\State Given function $f$, $A\in\C^{N\times N}$, $\vb\in\C^N$ of norm $\beta>0$, tolerances $\varepsilon,\delta >0$, and   integers $m,k_{\max}\geq 1$ 
		\State Compute an Arnoldi decomposition for $\spK_m(A,\vb)$: $A V_{m}^\one =  V_{m}^\one H_m^\one + h_{m+1,m}^\one\vv_{m+1}^\one \ve_m^T$
		\State Compute first Arnoldi approximation $\vf_m^\one \leftarrow  \beta V_m^\one f(H_m^\one)\ve_1$
		\While {$k = 2, \ldots, k_{\max}$ and error measure greater than $\varepsilon$}
			\State Compute an Arnoldi decomposition for $\spK_m \left(A,\vv_{m+1}^\kless \right)$: $A V_{m}^\k = V_{m}^{(k)} H_m^{(k)} + h_{m+1,m}^{(k)}\vv_{m+1}^{(k)} \ve_m^T$
			\State Compute $I_m^\kless \left(H_m^\k \right) \ve_1$ defined by \eqref{eq:error_integral} up to tolerance $\delta$ using adaptive quadrature
		\State Update $\vf_m^\k \leftarrow \vf_m^\kless + V_m^\k I_m^\kless \left(H_m^\k \right) \ve_1$
		\EndWhile
	\end{algorithmic}
\end{algorithm}

A key benefit of having adaptive quadrature in Algorithm~\ref{alg:restarts_mat_func2} is that the amount of work might reduce as we converge.  Note that if the quadrature rule is fixed a priori and throughout the algorithm, then Algorithm~\ref{alg:restarts_mat_func2} can be written more efficiently with a fixed rational approximation; see, e.g., \cite{HaleHighamTrefethen2008}.

\subsection{Software}\label{sec:rest3}
Software packages implementing Arnoldi-based algorithms for matrix functions, predominantly for the matrix exponential, are listed below.
\begin{itemize}
	\item EXPOKIT: \url{https://www.maths.uq.edu.au/expokit/}. MATLAB and FORTRAN implementations of  Arnoldi-based methods  for the matrix exponential and related functions~\cite{Sidje1998}.
	\item \texttt{expAtv}: \url{https://cran.r-project.org/web/packages/expm/}. An R implementation of the EXPOKIT methods.
		\item NEXPOKIT: \url{https://www.cs.purdue.edu/homes/dgleich/codes/nexpokit/}. Network matrix exponentials for link-prediction, centrality measures, and more.
	\item The Matrix Function Toolbox: \url{https://www.maths.manchester.ac.uk/~higham/mftoolbox/}. Non-restarted Arnoldi method written in \textsc{Matlab}.
	\item \texttt{markovfunmv}: \url{http://guettel.com/markovfunmv/}. A \textsc{Matlab} implementation of a black-box rational Arnoldi method for Markov matrix functions.
	\item \texttt{expmARPACK}: \url{https://swmath.org/software/13396}. Krylov subspace exponential time domain solution of Maxwell's equations in photonic crystal modeling. 
\end{itemize}

\textsc{Matlab} software for restarted Krylov approximations to $f(A)\vb$ are hosted at the following websites:
\begin{itemize}
	\item \texttt{funm\_kryl}: \url{http://www.guettel.com/funm_kryl/}.  Algorithm~\ref{alg:restarts_mat_func1} and versions from papers related to \cite{EiermannErnst2006} are implemented.  There are also options for deflated and thick restarting; see \cite{EiermannErnstGuettel2011}.
	\item \texttt{funm\_quad}: \url{http://www.guettel.com/funm_quad/}.  Algorithm~\ref{alg:restarts_mat_func2} is implemented for various Cauchy--Stieltjes functions, as well as the exponential function.   The code accompanies the papers \cite{FrommerGuettelSchweitzer2014a, FrommerGuettelSchweitzer2014b} and also allows for thick restarts.
	\item \texttt{B(FOM)$^2$}: \url{https://gitlab.com/katlund/bfomfom-main}.  Block Krylov methods for computing the action of $f(A)$ to multiple vectors simultaneously are implemented in the spirit of Algorithm~\ref{alg:restarts_mat_func2}.  The syntax mimics that of \texttt{funm\_quad}, and the code accompanies the papers \cite{FrommerLundSzyld2017, FrommerLundSzyld2020}.
\end{itemize}

\section{Summary} \label{sec:summary}

We have provided an overview of essential methods for approximating $f(A)\vb$ when $A$ is large-scale, i.e., in cases when only the action of $A$ on vectors is feasible to compute and when available memory resources are a limitation. Despite what may seem like severe restrictions, two key approaches---(a priori) expansion-based methods and (restarted) Krylov methods---have emerged over the years and have given rise to a variety of excellent algorithms. As polynomials form the backbone of both approaches, these methods are easily accessible to practitioners and they can be analyzed using  classical tools from approximation theory. 

\changed{In terms of open research questions, we just point here to the need for more accurate stopping criteria, stability analysis, preconditioning, and high-performance software implementations for restarted Krylov methods, as outlined in Section~\ref{sec:open_problems}. We hope that this manuscript will provide a good launching point for interested researchers to continue exploring these challenging problems. Also, a clearer understanding of which method to use under what circumstances would be very welcome. G\"{u}ttel and Schweitzer \cite{GuettelSchweitzer2020} make a first step in this direction by comparing the expected performance of the restarted Lanczos method, a multi-shift CG method, as well as two inexact rational Krylov methods with polynomial inner solves for the approximation of Stieltjes functions of Hermitian matrices. By comparing error bounds for these methods, it is found that both restarted Lanczos and multi-shift CG are among the best available polynomial methods for this purpose in the limited-memory scenario. Further work should address the question of how methods compare when implemented in practice, using, e.g., parallel computations and mixed-precision arithmetic.}
\section{Open problems} \label{sec:open_problems}

General stopping criteria for restarted Krylov methods remain an open issue. 
The error bounds provided by Eiermann and Ernst \cite{EiermannErnst2006}, as well as Frommer et\ al. \cite{FrommerGuettelSchweitzer2014a, FrommerGuettelSchweitzer2014b}, are known to be  pessimistic in practice. Some progress for Cauchy--Stieltjes functions of Hermitian positive definite matrices  has been made by Frommer and Schweitzer \cite{FrommerSchweitzer2016}. Using the integral form of the error~\eqref{eq:error_integral} and the fact that it can again be written as a Cauchy--Stieltjes function, they compute tighter lower and upper bounds of $\norm{f(A) \vb - \vf_m}_2$, where $\vf_m$ is the non-restarted Lanczos approximation, via connections to bilinear forms and Gauss quadrature from Golub and Meurant's monograph \cite{GolubMeurant2010}.  The bounds can also be computed via quadrature at little additional cost, making them useful in practice.

Preconditioning Krylov approximations for $f(A)\vb$ also remains challenging.  Some success has been seen for the matrix exponential; see, e.g., van den Eshof and Hochbruck \cite{EshofHochbruck2006} and Wu et\ al. \cite{WuPangSun2018}.  Connected to the preconditioning challenge is a rigorous backward stability analysis for Algorithms~\ref{alg:restarts_mat_func1} and \ref{alg:restarts_mat_func2}.  Numerical experiments for moderately conditioned $A$ indicate that the algorithms are stable, but without a universal preconditioning mechanism, there are bound to be scenarios in which these algorithms struggle.  Furthermore, restarting introduces other sources of instability, and it is well known that both restarted FOM and GMRES can stagnate \cite{DuintjerTebbensMeurant2012, DuintjerTebbensMeurant2014, Schweitzer2016a}.  One possible way to overcome stagnation is to allow for variable basis sizes, but criteria for determining an optimal basis size at each restart cycle remain open. The residual-based restarting method by Botchev and Knizhnerman \cite{BotchevKnizhnerman2020} does allow for variable basis sizes by exploiting the existence of an initial value problem that is satisfied by the matrix exponential and closely related functions, but it is unclear how to generalize this idea to other functions.

High-performance implementations of Algorithms~\ref{alg:restarts_mat_func1} and \ref{alg:restarts_mat_func2}, particularly parallelizable and communication-reducing formulations of these algorithms, are needed.  Much of the theory for Krylov methods for linear systems should transfer (see, e.g., $s$-step \cite{BallardCarsonDemmel2014} or enlarged methods \cite{GrigoriMoufawadNataf2016}), but there are  difficulties regarding memory allocation for the growing Hessenberg matrices and the adaptive quadrature rules that need to be overcome.

While not touched on here in detail, we finally point out that there are many similarities between techniques used for matrix equations and those used for matrix functions, and success for one class of problems often transfers to the other.  Some preliminary work to connect the two has been conducted by Kressner \cite{Kressner2019}, and a successful application of restarts to Krylov methods for matrix equations has been developed in Kressner et\ al.\cite{KressnerLundMassei2020}.
\begin{paragraph}{Acknowledgements.}
Parts of this survey were prepared during visits of the last two authors to the Department of Mathematics at The University of Manchester, whose hospitality is gratefully acknowledged. The first author acknowledges support by The Alan Turing Institute under the EPSRC grant~EP/N510129/1.   The third author was supported in part by the SNSF research project \emph{Fast Algorithms from Low-rank Updates}, grant number 200020\_178806, and in part by the Charles University PRIMUS grant, \emph{Scalable and Accurate Numerical Linear Algebra for Next-Generation Hardware}, project ID PRIMUS/19/SCI/11.
\end{paragraph}

\bibliographystyle{abbrv}
\bibliography{mat_func_survey}

\end{document}